\documentclass[a4paper,12pt]{article}
\usepackage[utf8]{inputenc}

\usepackage{mathrsfs}
\usepackage{graphicx}
\usepackage{enumerate}
\usepackage{multicol}
\usepackage{color}
\usepackage{amsmath,amssymb,amscd}
\usepackage{amsthm}
\usepackage{mathabx}

\usepackage{pdfpages}
\usepackage{hyperref}

\usepackage{times}
\usepackage[varg]{txfonts}

\usepackage[top=1in, bottom=1in, left=0.75in, right=0.75in]{geometry}

\theoremstyle{plain}
\newtheorem{thm}{Theorem}

\theoremstyle{definition}

\newtheorem{ex}{Example}

\theoremstyle{remark}

\newtheorem*{ackn}{Acknowledgment}

\newcommand{\C}{\mathbb{C}}
\newcommand{\R}{\mathbb{R}}
\newcommand{\N}{\mathbb{N}}

\title{Endomorphisms and derivations of the measure algebra of commutative hypergroups}
\author{{\.Z}ywilla Fechner, Eszter Gselmann and L\'{a}szl\'{o} Sz\'ekelyhidi}

\begin{document}
 
 \maketitle
 
	\begin{abstract}
		Endomorphisms of the measure algebra of commutative hypergroups are investigated. We focus on derivations and higher order derivations which are closely related to moment function sequences of higher rank. We describe the exact connection between those higher order derivations which are endomorphisms of the measure algebra if it is considered as a module over the ring of continuous functions. 
	\end{abstract}

\section{Introduction}\label{intro}
Moment function sequences on hypergroups have been defined in \cite{MR1312826} (see also \cite{MR2107959,MR2161803,MR2431934,MR2978690}). In \cite{FecGseSze20} and \cite{FecGseSze22} moment function sequences of higher rank have been defined and investigated. In this paper we present results which explain a connection between moment function sequences of higher rank and higher order derivations of the measure algebra of commutative hypergroups. For the discussion on derivations on abstract structures we refer to \cite{GseKisVin18} or
\cite{KisLac18}
It turns out that moment function sequences generate higher order derivations, which are not just linear operators of the measure algebra, but they are endomorphisms of the measure algebra when considered as a module over the ring of continuous functions. And conversely: every higher order derivation of this property generates a moment function sequence of higher rank. 
\section{Measure algebra}
Let $X$ be a commutative hypergroup. The space of all complex valued continuous functions on $X$ will be denoted by $\mathcal C(X)$. Equipped with the pointwise linear operations and with the topology of uniform convergence on compact sets it is a locally convex topological vector space. We recall that its topological dual $\mathcal C^*(X)$ can be identified with the {\it measure algebra} of $X$: this is the space $\mathcal M_c(X)$ of all compactly supported complex Borel measures on $X$ equipped with the linear operations (addition and multiplication by complex numbers) and with the convolution of measures defined by
$$
\int_X f\,d(\mu*\nu)=\int_X\int_X f(x*y)\,d\mu(x)\,d\nu(y)
$$ 
for each function $f$ in $\mathcal C(X)$. If we equip $\mathcal M_c(X)$ with weak*-topology and then it is a locally convex topological space, in fact, a topological algebra. The topological dual of $\mathcal M_c(X)$ can be identified with $\mathcal C(X)$: for each weak*-continuous linear functional $\Lambda:\mathcal M_c(X)\to \C$ there exists a unique continuous   function $\varphi$ in $\mathcal C(X)$ such that
$$
\Lambda(\mu)=\int_X \varphi\,d\mu
$$
holds whenever $\mu$ is in $\mathcal M_c(X)$ (see \cite{MR0365062}, Section 3.14). We can write $\varphi(\mu)=\Lambda(\mu)$. In particular, we have  $\varphi(\delta_x)=\varphi(x)$ for each $x$ in $X$. For instance, the multiplicative functionals of the measure algebra arise in this manner exactly from the exponentials of $X$.

Besides the linear structure of the measure algebra we can consider it as a module over the ring $\mathcal C(X)$ of continuous functions on $X$: the action of $\varphi$ in $\mathcal C(X)$ on $\mathcal M_c(X)$ is defined by the multiplication of the measure $\mu$ by $\varphi$ defined as
$$
\langle \varphi\cdot \mu, f\rangle=\int_X f \varphi\,d\mu.
$$ 
Clearly, $\varphi \cdot \mu$ is in $\mathcal C(X)$. We simply write $\varphi \mu$ for $\varphi\cdot \mu$.

If we say that a map $F:\mathcal M_c(X)\to \mathcal M_c(X)$ is a {\it module homomorphism}, then we mean that it is an endomorphism of $\mathcal M_c(X)$ {\it as a $\mathcal C(X)$-module}. In particular, we have $F(\varphi\mu)=\varphi  F(\mu)$ for each $\mu$ in $\mathcal M_c(X)$ and $\varphi$ in $\mathcal C(X)$. Obviously, this implies homogeneity with respect to multiplication with complex numbers, which can be identified with constant functions. A module homomorphism is called {\it continuous}, if it is continuous with respect to the weak*-topology. A {\it multiplicative module homomorphism} is a module homomorphism, which is also a homomorphism of the algebra structure of $\mathcal M_c(X)$, that is, it preserves convolution. 

A {\it derivation} of the measure algebra usually defined as a continuous linear operator $D:\mathcal M_c(X)\to\mathcal M_c(X)$ satisfying the additional property 
\begin{equation}\label{der2}
D(\mu*\nu)=D\mu*\nu+\mu*D\nu
\end{equation}
for each $\mu,\nu$ in $\mathcal M_c(X)$. In this paper, however, we shall consider a modified version of this concept. Namely, instead of linearity we require that $D$ is a continuous module homomorphism of $\mathcal M_c(X)$ as a module over the ring of continuous functions $\mathcal C(X)$. In other words, besides \eqref{der2} we assume that 
\begin{equation}\label{der1}
	D(\mu+\nu)=D\mu+D\nu
\end{equation}
and
\begin{equation}\label{der3}
	D(\varphi \mu)=\varphi D\mu
\end{equation}
holds for each $\mu,\nu$ in $\mathcal M_c(X)$ and $\varphi$ in $\mathcal C(X)$. The latter two properties express the fact that $D$ is a module homomorphism.

Suppose that $r$ is a positive integer. We say that tha family $(D_{\alpha})_{\alpha\in\N^r}$ of continuous module homomorphisms of $\mathcal M_c(X)$ is a {\it higher order derivation of rank $r$} if for each $\alpha$ in $\N^r$ we have
\begin{equation}\label{hider}
	D_{\alpha}(\mu*\nu)=\sum_{\beta\leq \alpha} \binom{\alpha}{\beta} D_{\beta}\mu*D_{\alpha-\beta}\nu
\end{equation}
whenever $\mu,\nu$ is in $\mathcal M_c(X)$. The {\it order} of $D_{\alpha}$ is defined as $|\alpha|$. In particular, $D_0$ satisfies
\begin{equation}\label{alghom}
D_0(\mu*\nu)=D_0\mu*D_0\nu,
\end{equation}
that is, $D_0$ is an algebra homomorphism and module homomorphism simultaneously. We call such a mapping a {\it multiplicative module homomorphism}. Moreover, for each $\alpha$ with $|\alpha|=1$ we have
\begin{equation}\label{gender}
D_{\alpha}(\mu*\nu)=D_0\mu*D_{\alpha}\nu+D_{\alpha}\mu*D_0\nu
\end{equation}
holds. In the case $D_0=id$, the identity homomorphism, then $D_{\alpha}$ is a derivation. In general, a continuous module homomorphism $D$ satisfying \eqref{gender} in place of $D_{\alpha}$ with a given $D_0$ which is a multiplicative algebra homomorphism, is called a {\it $D_0$-derivation}. Thus ordinary derivations are exactly the $id$-derivations.
\section{Module homomorphisms}
We begin with the characterization of continuous module homomorphism on $\mathcal{M}_c(X)$.
\begin{thm}\label{mod}
	Let $X$ be a commutative hypergroup. Every continuous module homomorphism $F$ on $\mathcal M_c(X)$ has the form
	\begin{equation}\label{modhom}
		\langle F(\mu),f\rangle = \int_X f\cdot \varphi \, d\mu
	\end{equation}
with a unique function $\varphi$ in $\mathcal C(X)$. Conversely, each  mapping of the form given by \eqref{modhom} is a module homomorphism of $\mathcal M_c(X)$.
\end{thm}

\begin{proof}
	Given a function $\varphi$ in $\mathcal C(X)$ we define the mapping $F: \mathcal M_c(X)\to\mathcal M_c(X)$ by
	$$
	\langle F (\mu),f\rangle =\int_X f\cdot  \varphi \,d\mu=\langle \varphi \mu,f\rangle
	$$
	for each $\mu$ in $\mathcal M_c(X)$ and $f$ in $\mathcal C(X)$. Clearly, $F$ is additive and weak*-continuous. Moreover, for each $\psi$ in $\mathcal C(X)$ we have
	$$
	\langle F (\psi \mu),f\rangle=\int_X f\cdot \varphi\,d(\psi \mu)=\int_X f \varphi \psi\, d\mu=\int_X f  \psi\cdot \varphi\, d\mu=
	$$
	$$
	\int_X f \psi\,d(F(\mu))=\int_X f \,d(\psi F(\mu))=\langle \psi F(\mu),f\rangle,
	$$
	hence 
	$$
	F (\psi \mu)=\psi F(\mu),
	$$	
	which proves that $F$ is a continuous module homomorphism.
	\vskip.2cm
	
	For the converse, let $F:\mathcal M_c(X)\to\mathcal M_c(X)$ be a continuous module homomorphism. 
We define $\Lambda: \mathcal M_c(X)\to\C$ by
	$$
	\Lambda(\mu)=\langle F(\mu),1\rangle
	$$
	for each $\mu$ in $\mathcal M_c(X)$. Clearly, $\Lambda$ is a linear mapping from $\mathcal M_c(X)$ into $\C$. We show that $\Lambda$ is also weak*-continuous. Indeed, if $(\mu_i)_{i\in I}$ is a net converging to the zero measure in the weak*-topology, then the net $(F(\mu_i))_{i\in I}$  converges to the zero measure n the weak*-topology as well, by the weak*-continuity of $F$. It follows that the net $(\langle F(\mu_i),1\rangle)_{i\in I}$ converges to zero -- in other words, the net $(\Lambda(\mu_i))_{i\in I}$ converges to zero, hence $\Lambda$ is weak*-continuous. We infer that $\Lambda$ is a linear functional of the space $\mathcal M_c(X)$. It is known, that $\Lambda$ arises from a function in $\mathcal C(X)$ -- more exactly, there exists a unique continuous function $\varphi:X\to\C$ such that
	$$
	\Lambda(\mu)=\int_X \varphi\,d\mu
	$$
	holds for each $\mu$ in $\mathcal M_c(X)$. 
	Then we conclude for each $f$ in $\mathcal C(X)$:
	$$
	\int_X f\,d(\varphi \mu)=\int_X f\varphi\,d\mu=\int_X \varphi \,d(f \mu)= \Lambda(f \mu)=\langle F(f \mu),1\rangle=
	$$
	$$
	\langle f F(\mu),1\rangle=\int_X 1\,d (f F(\mu))=\int_X f\,dF(\mu),
	$$
	that is $F(\mu)=\varphi \mu$, which is equation \eqref{modhom}.

\end{proof}

We infer that the general form of a continuous endomorphism of the measure algebra, as a $\mathcal C(X)$-module, is exactly multiplication by a continuous function. The following theorem describes the continuous multiplicative module homomorphisms of the measure algebra.

\begin{thm}\label{multmod}
	Let $X$ be a commutative hypergroup. Every continuous nonzero multiplicative module homomorphism $F$ on $\mathcal M_c(X)$ has the form
	\begin{equation}\label{multmodhom}
		\langle F(\mu),f\rangle= \int_X f \cdot m\, d\mu
	\end{equation}
	for some exponential $m$ in $\mathcal C(X)$. The exponential $m$ is uniquely determined by $F$. Conversely, each  mapping of the form given by \eqref{multmodhom} is a continuous multiplicative module homomorphism of $\mathcal M_c(X)$.
\end{thm}

\begin{proof}
	The mapping $F$ defined by \eqref{multmodhom} is a continuous module homomorphism of $\mathcal M_c(X)$, by the previous theorem. We show that if $m$ is an exponential,  the it is also multiplicative. Indeed, for each $\mu, \nu$ in $\mathcal M_c(X)$ and $\varphi$ in $\mathcal C(X)$, we have
	$$
	\langle F(\mu*\nu),\varphi\rangle =\int_X m \varphi\,d(\mu*\nu)=\int_X \int_X m(x*y)\varphi(x*y)\,d\mu(x)\,d\nu(y)=
	$$
	$$
	\int_X  \int_X \varphi(x*y)\,d(m\cdot \mu)(x)\,d(m\cdot \nu)(y)  =\int_X  \int_X \varphi(x*y)\,dF(\mu)(x)\,dF(\nu)(y)=
	$$
	$$
	 \int_X \varphi \,d(F(\mu)*F(\nu))=\langle F(\mu)*F(\nu),\varphi\rangle.
	$$
	
	For the converse, let $F$ be a nonzero continuous multiplicative module homomorphism of $\mathcal M_c(X)$. Then, by the previous theorem, it has the form
	$$
	F(\mu)=\varphi \cdot \mu 
	$$
	with some continuous function $\varphi:X\to \C$. For $x,y$ in $X$ we have, by the multiplicativity of $F$:
	$$
\varphi(x*y)=\int_X \varphi(z)\,d(\delta_x*\delta_y)(z)=\langle F(\delta_x*\delta_y),1\rangle=\langle F(\delta_x)*F(\delta_y),1\rangle=
	$$
	$$
	\int_X \,d(F(\delta_x)*F(\delta_y))=\int_X \,dF(\delta_x)\cdot  \int_X \,dF(\delta_y)=
	$$
	$$
	\int_X \varphi\,d\delta_x \cdot \int_X f\,d\delta_y=\varphi(x) \varphi(y).
	$$
	It follows $\varphi(o)=\varphi(o)\varphi(o)$. If $\varphi(o)=0$, then $\varphi=0$, and $F=0$, which is excluded. Hence $\varphi(o)=1$, and $\varphi$ is an exponential. Our theorem is proved.
\end{proof}

\section{Derivations}
In this section we describe the connection between higher order derivation and generalized moment function sequences.
\begin{thm}\label{deri}
Let $X$ be a commutative hypergroup and $r$ a positive integer. The family $(D_{\alpha})_{\alpha\in\N^r}$ of self-mappings on $\mathcal M_c(X)$ is a  continuous higher order derivation of order $r$ if and only if there exists a generalized moment function sequence  $(\varphi_{\alpha})_{\alpha\in\N^r}$ of rank $r$ such that
	\begin{equation}\label{dermom}
		\langle D_{\alpha}\mu, f\rangle=\int_X f \cdot \varphi_{\alpha}\,d\mu
	\end{equation}
holds for each $\mu$ in $\mathcal M_c(X)$, $f$ in $\mathcal C(X)$ and $\alpha$ in $\N^r$.
\end{thm}

\begin{proof}
Let $(\varphi_{\alpha})_{\alpha\in\N^r}$ be  a generalized moment function sequence of rank $r$ on $X$, and we define $D_{\alpha}\mu$ for each $\mu$ in $\mathcal M_c(X)$, $f$ in $\mathcal C(X)$ and $\alpha$ in $\N^r$ by equation \eqref{dermom}. Obviously, $D_{\alpha}:\mathcal M_c(X)\to\mathcal M_c(X)$ is a linear mapping for each $\alpha$ in $\N^r$. The weak*-continuity of $D_{\alpha}$ follows easily from the continuity of the functions $\varphi_{\alpha}$.
\vskip.2cm

From Theorem \ref{mod}, it follows immediately, that $D_{\alpha}$ is a module homomorphism for each $\alpha$ in $\N^r$. Now we show that $(D_{\alpha})_{\alpha\in\N^r}$ is a  continuous higher order derivation of order $r$. Given $\alpha$ in $\N^r$, $\mu,\nu$ in $\mathcal M_c(X)$ and $f$ in $\mathcal C(X)$ we have
$$
\langle D_{\alpha}(\mu*\nu),f\rangle =\int_X f \cdot \varphi_{\alpha}\,d(\mu*\nu)=\int_X f(x*y) \cdot \varphi_{\alpha}(x*y)\,d\mu(x)\,d\nu(y)=
$$
$$
\sum_{\beta\leq\alpha} \binom{\alpha}{\beta}\int_X \int_X f(x*y)\varphi_{\beta}(x) \varphi_{\alpha-\beta}(y)\,d\mu(x)\,d\nu(y)=
$$
$$
\sum_{\beta\leq\alpha} \binom{\alpha}{\beta}\int_X \int_X f(x*y) \,d(\varphi_{\beta}\mu)(x)\,d(\varphi_{\alpha-\beta}\nu)(y)=
$$
$$
\sum_{\beta\leq\alpha} \binom{\alpha}{\beta}\int_X \int_X f(x*y) \,dD_{\beta}\mu(x)\,dD_{\alpha-\beta}\nu(y)=
$$
$$
\int_X  f\,d\Big( \sum_{\beta\leq\alpha} \binom{\alpha}{\beta} D_{\beta}\mu*D_{\alpha-\beta}\nu\Big)=\langle \sum_{\beta\leq\alpha} \binom{\alpha}{\beta} D_{\beta}\mu*D_{\alpha-\beta}\nu,f\rangle,
$$
which proves our statement.
\vskip.2cm

Now we prove the converse statement. Let $(D_{\alpha})_{\alpha\in\N^r}$ be a  continuous higher order derivation of order $r$ on $\mathcal M_c(X)$. We define
$$
\varphi_{\alpha}(x)=\langle D_{\alpha}\delta_x,1\rangle
$$
for each $\alpha$ in $\N^r$ and $x$ in $X$. If $x$ tends to $x_0$ in $X$, then $\delta_x$ converges to $\delta_x{x_0}$ in the weak*-topology, hence, by the weak*-continuity of $D_{\alpha}$, $\varphi_{\alpha}(x)$ tends to $\varphi_{\alpha}(x_0)$. This proves the continuity of $\varphi_{\alpha}$. On the other hand, for each $\alpha$ in $\N^r$ and $x,y$ in $X$ we have
$$
\varphi_{\alpha}(x*y)=\langle D_{\alpha}\delta_{x*y},1\rangle=\langle D_{\alpha}(\delta_{x}*\delta_{y}),1\rangle=
$$
$$
 \sum_{\beta\leq\alpha} \binom{\alpha}{\beta}\langle D_{\beta}\delta_{x}*D_{\alpha-\beta}\delta_{y},1\rangle= \sum_{\beta\leq\alpha} \binom{\alpha}{\beta}\int_X 1\,d(D_{\beta}\delta_{x}*D_{\alpha-\beta}\delta_{y})=
$$
$$
\sum_{\beta\leq\alpha} \binom{\alpha}{\beta}\int_X \int_X  1\,dD_{\beta}\delta_{x}\,dD_{\alpha-\beta}\delta_{y}=\sum_{\beta\leq\alpha} \binom{\alpha}{\beta}\int_X  1\,dD_{\beta}\delta_{x}\cdot \int_X 1\,dD_{\alpha-\beta}\delta_{y}=
$$
$$
\sum_{\beta\leq\alpha} \binom{\alpha}{\beta}\langle D_{\beta}\delta_x,1\rangle \cdot \langle D_{\alpha-\beta}\delta_{y},1\rangle=\sum_{\beta\leq\alpha} \binom{\alpha}{\beta}\langle \varphi_{\beta}(x) \cdot  \varphi_{\alpha-\beta}(y),
$$
and our theorem is proved.
\end{proof}

\section{Derivations of the Fourier algebra}
Let $X$ be a commutative hypergroup. We denote by $\tilde{X}$ the set of all exponentials \hbox{on $X$.} For each measure $\mu$ in $\mathcal M_c(X)$, we define the {\it Fourie--Laplace transform of $\mu$} as the function $\hat{\mu}:\tilde{X}\to\C$ by
$$
\hat{\mu}(m)=\int_X \check{m}\,d\mu.
$$
The set of all Fourier--Laplace transforms is called the {\it Fourier algebra} of $X$ and it is denoted by $\mathcal A(X)$. It is well-known that the Fourier--Laplace transform is bijective between $\mathcal M_c(X)$ and $\mathcal A(X)$. This makes it possible to equip a topology on $\mathcal A(X)$ in the obvious way, that makes the bijection between $\mathcal M_c(X)$ and $\mathcal A(X)$ a homeomorphism. Using this bijection we can study derivations on the Fourier algebra as well. First we consider the Fourier algebra a module over the ring $\mathcal C(X)$ in the following way: for each $\varphi$ in $\mathcal C(X)$ and $\mu$ in $\mathcal M_c(X)$ we define
$$
\varphi \cdot \hat{\mu}=(\varphi \mu)\,\hat{}.
$$ 
With this operation $\mathcal A(X)$ is a module over $\mathcal C(X)$. A module homomorphism of $\mathcal A(X)$ into itself is an additive mapping $F:\mathcal A(X)\to\mathcal A(X)$ such that
$$
F(\varphi\cdot \hat{\mu})=\varphi\cdot F(\hat{\mu})
$$ 
holds for each $\mu$ in $\mathcal A$ and $\varphi$ in $\mathcal C(X)$. Accordingly, the family of module homomorphisms $(\partial_{\alpha})_{\alpha\in\N^r}$ is  a higher order derivation of rank $r$, if it satisfies
\begin{equation}\label{Fouder}
\partial_{\alpha}(\hat{\mu}\cdot \hat{\nu})=\sum_{\beta\leq\alpha} \binom{\alpha}{\beta} \partial_{\beta}\hat{\mu}\cdot \partial_{\alpha-\beta}\hat{\nu}
\end{equation}
for each $\mu,\nu$ in $\mathcal M_c(X)$. The following theorem is obvious.

\begin{thm}\label{Fouderchar}
Let $X$ be a commutative hypergroup, and $r$ a positive integer. Then the family $(D_{\alpha})_{\alpha\in\N^r}$ is a higher order derivation of rank $r$ on $\mathcal M_c(X)$ if and only if the family of mappings $\hat{D}_{\alpha}:\mathcal A\to\mathcal A$ defined by 
\begin{equation}\label{trFou}
	\hat{D}_{\alpha}\hat{\mu}=(D_{\alpha}\mu)\,\hat{}
\end{equation}
for each $\alpha$ in $\N^r$ and $\mu$ in $\mathcal M_c(X)$ is a higher order derivation of rank $r$ on $\mathcal A(X)$. 
\end{thm}

\begin{proof}
	The only thing to show is that $D_{\alpha}$ and $\hat{D}_{\alpha}$ are module homomorphisms simultaneously. If $D_{\alpha}$ is, then have for each $\varphi$ in $\mathcal C(X)$:
	$$
	\hat{D}_{\alpha}(\varphi\cdot \hat{\mu})=	\hat{D}_{\alpha}(\varphi \mu)\,\hat{}=(D_{\alpha}(\varphi\cdot \mu))\,\hat{}=(\varphi\cdot D_{\alpha}\mu)\,\hat{}=\varphi\cdot (D_{\alpha} \mu)\,\hat{}=\varphi\cdot \hat{D}_{\alpha}\hat{\mu},
	$$
	hence $\hat{D}_{\alpha}$ is a module homomorphism, too. Conversely, if $\hat{D}_{\alpha}$ is a module homomorphism, then for each $\varphi$ in $\mathcal C(X)$ we have:
	$$
	(D_{\alpha}(\varphi\cdot \mu))\,\hat{}=	\hat{D}_{\alpha}(\varphi \mu)\,\hat{}=\hat{D}_{\alpha}(\varphi\cdot \hat{\mu})=\varphi\cdot \hat{D}_{\alpha}\hat{\mu}=
	$$
	$$
\varphi\cdot (D_{\alpha} \mu)\,\hat{}=(\varphi\cdot D_{\alpha}\mu)\,\hat{}\,,
	$$
	which implies $D_{\alpha}(\varphi\cdot \mu)=\varphi\cdot D_{\alpha}\mu$, by the injectivity of the Fourier--Laplace transform. The equivalence of the derivation properties can be verified by a simple calculation, hence the theorem is proved.
\end{proof}

\section{Examples}

In this section we present several examples of applications of our results.

\begin{ex}
Let $G=\R$, and we consider the functions $\varphi_{k,\lambda}(x)=x^k e^{\lambda x}$ for $k=0,1,\dots$, where $x$ is in $\R$ and $\lambda$ is in $\C$. Clearly, they form a generalized moment sequence of rank $1$ as
$$
\varphi_{n,\lambda}(x+y)=(x+y)^n e^{\lambda (x+y)}=\sum_{k=0}^n \binom{n}{k} x^k e^{\lambda x} y^{n-k} e^{\lambda y}=\sum_{k=0}^n \binom{n}{k} \varphi_{k,\lambda}(x)\varphi_{n-k,\lambda}(y).
$$
This moment function sequence generates the following higher order derivation on the space $\mathcal M_c(\R)$:
$$
\langle D_k\mu, f\rangle= \int_{\R} x^k e^{\lambda x}f(x)\,d\mu(x). 
$$ 
We have 
$$
\langle D_k(\psi \mu), f\rangle= \int_{\R} x^k e^{\lambda x}f(x)\cdot \psi(x)\,d\mu(x) =\int_{\R} x^k e^{\lambda x}f(x) \,d(\psi\mu)(x)=\langle \psi D_k( \mu), f\rangle
$$ 
hence $D_k$ is a module homomorphism for each $k=0,1,\dots$. In particular, for $\lambda=0$
$$
\langle D_k\mu, 1\rangle= \int_{\R} x^k \,d\mu(x), 
$$ 
which corresponds to the classical moments of the measures.

\end{ex}
\vskip.2cm

\begin{ex}
Now let $X=D(\theta)$ be the two-point hypergroup on the set $\{0,1\}$ with \hbox{$0<\theta\leq 1$}, where $0$ is the identity. We have two exponentials on this hypergroup: the identically $1$ function: $m_0\equiv 1$, and the function
$$
m_1(x)=
\begin{cases}
	1&\enskip\text{for}\enskip x=0\\
	-\theta&\enskip\text{for}\enskip x=1.
\end{cases}
$$
The only generalized moment sequence of any rank with $\varphi_0$ is trivial: $\varphi_{\alpha}\equiv 0$ for $|\alpha|>0$. Accordingly, the only higher order derivation on $\mathcal M_c(X)$ is the trivial, $D_{\alpha}=0$ for $|\alpha|>0$.
\end{ex}
\vskip.2cm

\begin{ex}
For a less trivial example we consider the polynomial hypergroup $X$ generated by the sequence of polynomials $(P_n)_{n\in\N}$. Let $z$ be arbitrary in $\C$ and we define $\varphi_k(n)=P_n^{(k)}(z)$ for $k,n=1,2,\dots$, then $(\varphi_k)_{k\in\N}$ is a moment function sequence. Indeed, 
$$
\varphi_k(n*m)=\sum_{l=|n-m|}^{n+m} c(m,n,l)\varphi_k(l)= \sum_{l=|n-m|}^{n+m} c(m,n,l)P_l^{(k)}(z)=
$$
$$
\big(P_n(z)\cdot P_m(z)\bigr)^{(k)}=
\sum_{j=0}^k \binom{k}{j} P_n^{(j)}(z)\cdot P_m^{(k-j)}(z)=\sum_{j=0}^k \binom{k}{j} \varphi_j(n) \varphi_{k-j}(m).
$$

By Theorem \ref{deri}, this moment function sequence generates the higher order derivation $(D_{k})_{k\in\N}$ of rank one defined by
$$
\langle D_k\mu,f\rangle= \int f(n)\cdot P_n^{(k)}(z)\,d\mu(n) 
$$
for each $k=0,1,\dots$ and $f:\N\to \C$. Observe, that 
$$
\langle D_k\mu,1\rangle=\int_X P_n^{(k)}(z)\,d\mu(n)=\hat{\mu}^{(k)}(z),
$$
that is, the higher order derivation $(D_{k})_{k\in\N}$ corresponds to the higher order differentiation of the Fourier--Laplace transform with respect to the parameter $\lambda$, which represents the exponential $n\mapsto P_n(z)$. Let $z=0$, and let $(D_k)_{k\in\N}$ denote the higher order derivation which corresponds to the moment function sequence $\varphi_k(n)=P_n^{(k)}(0)$. Observe that, for each $\mu$ in $\mathcal M_c(X)$, $\hat{\mu}$ is a polynomial. Hence we have for each $\lambda$ \hbox{in $\C$:}
$$
\hat{\mu}(\lambda)=\sum_k \frac{\hat{\mu}^{(k)}(0)}{k!} \lambda^k=\sum_k \frac{\lambda^k}{k!} \int_X P_n^{(k)}(0)\,d\mu(n)=\sum_k \frac{\lambda^k}{k!} \langle D_k\mu,1\rangle,
$$
that is, the higher order derivation determines $\hat{\mu}$ for each $\mu$. This can be considered as a kind of "Taylor formula".
\end{ex}

\begin{ackn}
 The research of E.~Gselmann has partially been carried out with the help of the project 2019-2.1.11-T\'{E}T-2019-00049,
which has been implemented with the support provided from NRDI (National Research, Development
and Innovation Fund of Hungary), financed under the T\'{E}T funding scheme.
\\
The research of E.~Gselmann and L.~Sz\'{e}kelyhidi has been supported by the NRDI (National Research, Development
and Innovation Fund of Hungary) Grant no. K 134191.
\end{ackn}

\end{document}